\documentclass[a4paper,oneside,11pt, reqno]{amsart}
\usepackage{amsfonts,amssymb,amsmath,amsthm}
\usepackage{mathrsfs}
\usepackage{url}
\usepackage{enumerate}
\usepackage{xcolor}
\usepackage{hyperref}
\usepackage{todonotes}
\usepackage{ulem} 

\newtheorem{theorem}{Theorem}[section]
\newtheorem*{theorem*}{Theorem}
\newtheorem{lemma}[theorem]{Lemma}

\newtheorem{prop}[theorem]{Proposition}
\newtheorem{cor}[theorem]{Corollary}
\theoremstyle{definition}

\newtheorem{rem}[theorem]{Remark}

\author[S. Ghara]{Soumitra Ghara}
\author[R. Gupta]{Rajeev Gupta}
\author[M. R. Reza]{Md. Ramiz Reza}

\address[S. Ghara]{Department of Mathematics\\
Indian Institute of Technology Kharagpur\\
Midnapore-721302}

\address[R. Gupta]{School of Mathematics and Computer Science\\
Indian Institute of Technology Goa\\
Goa -  403401}

\address[Md. R. Reza]{School of Mathematics\\
Indian Institute of Science Education and Research Thiruvananthapuram \\
Kerala - 695551}

\email[S. Ghara]{soumitra@maths.iitkgp.ac.in}
\email[R. Gupta]{rajeev@iitgoa.ac.in}
\email[M. R. Reza]{ramiz.md@gmail.com}


\begin{document}
\title [Ces{\`a}ro  summability of Taylor series]{Ces{\`a}ro  summability of Taylor series in higher order weighted Dirichlet type spaces}

 \maketitle
 \begin{center}
 {\textit{Dedicated to Professor Jan Stochel on the occasion of his 70th anniversary}}
 \end{center}

\begin{abstract}
For a  positive integer $m$ and a finite non-negative Borel measure $\mu$ on the unit circle, we study the Hadamard multipliers of higher order weighted Dirichlet-type spaces $\mathcal H_{\mu, m}$. We show that if $\alpha>\frac{1}{2},$ then for any $f$ in $\mathcal H_{\mu, m},$ the sequence of generalized Ces{\`a}ro sums $\{\sigma_n^{\alpha}[f]\}$ converges to $f$. We further show that if $\alpha=\frac{1}{2}$ then for the Dirac delta measure supported at any point on  the unit circle, the previous statement breaks down for every positive integer $m$.
\end{abstract}

\section{Introduction}
The symbols $\mathbb T$ and $\mathbb D$ will denote the unit circle and the open unit disc in the complex plane $\mathbb C$ respectively. We use the symbols $\mathbb Z, \mathbb N$ and $\mathbb Z_{\geqslant 0}$ to denote the set of  integers, positive integers and non-negative integers respectively. The notation $\mathcal M_{+}(\mathbb T)$ stands for the set of all finite non-negative Borel measures on $\mathbb T$. Let $\mathcal O(\mathbb D)$ denote the space of all complex valued holomorphic functions on $\mathbb D.$ For a holomorphic function $f\in\mathcal O(\mathbb D),$ which has a power series representation of the form $f(z)= \sum_{k=0}^{\infty} a_k z^k,$  $z\in\mathbb D,$ the $n^{\it th}$ Taylor partial sum $s_n[f]$ and the $n^{\it th}$ Ces{\`a}ro sum $\sigma_n[f]$ are defined by 
\begin{align*}
s_n[f](z)& := \sum\limits_{k=0}^{n} a_k z^k,~~\mbox{and}\,\,\\
\sigma_n[f](z):=\frac{1}{n+1} \Big(\sum\limits_{k=0}^{n} s_k [f](z)\Big) &= \sum\limits_{k=0}^n\Big(1-\frac{k}{n+1}\Big)a_kz^k,\,\,\,\,~~n\in \mathbb Z_{\geqslant 0}.
\end{align*}

When exploring analytic function spaces, a common concern revolves around determining the density of the polynomial set within the given function space. Once the density of polynomials is established, the next natural question to address is how to construct polynomials that can closely approximate a given function.
In this context, it becomes natural to inquire whether, for a function $f$ residing in a normed function space, the sequence of Taylor partial sum $\{s_n[f]\}$ converge to $f$ in the norm associated with that space. 
It is worth noting that in several classical function spaces like Hardy space, Dirichlet space, and Bergman space defined over the unit disc $\mathbb D$, it is well-established that the sequence $\{s_n[f]\}$ indeed converges to $f$ as prescribed by the associated norm.
However, it is worth noting that there are instances where this convergence property does not hold. 
For example, it is well known that there exists a function $f$ in the disc algebra $A(\mathbb D)$ such that the sequence $\{s_n[f]\}$ does not converge to $f,$ see \cite[page 57]{Kat}. 
Another notable family of examples in the context of Hilbert function space is the family of weighted Dirichlet type spaces.  Richter introduced the weighted Dirichlet space $D(\mu)$ for each $\mu\in \mathcal M_{+}(\mathbb T)$ in order to study the structure of closed $M_z$-invariant subspaces of the classical Dirichlet space on $\mathbb D$ and to obtain a model for cyclic analytic $2$-isometries, see \cite{R}. 
Interestingly, when $\mu = \delta_{\lambda}$, the Dirac delta measure at a point $\lambda \in \mathbb T$, it has been established that there exists a function $f$ in $D(\delta_{\lambda})$ for which the sequence $\{s_n[f]\}$ does not converge to $f$ within the space $D(\delta_{\lambda})$, see \cite[Exercise 7.3 (2)]{Primer}. 
On the contrary, a remarkable discovery surfaced when Mashreghi and Ransford recently showed that for any arbitrary, but fixed, $\mu\in \mathcal M_+(\mathbb T)$, the sequence of Ces{\`a}ro sums $\{\sigma_n[f]\}$ converges to $f$ for every $f$ in the space $D(\mu)$ (refer to \cite[Theorem 1.6]{MasRans}). 
Furthermore, they refined their result by establishing that the sequence of generalized Ces{\`a}ro sums $\{\sigma_n^{\alpha}[f]\}$ also converge to $f$ for every $f$ in $D(\mu)$ and for each $\mu\in \mathcal M_+(\mathbb T)$ whenever $\alpha > 1/2$, see \cite[Theorem 1.1]{CesaroM}. 
For each $n\in \mathbb Z_{\geqslant 0},$ the generalized $n^{\it th}$-Ces{\`a}ro mean $\sigma_n^{\alpha}[f]$ is defined by 
\begin{align*}
\sigma_n^{\alpha}[f](z)= \binom{n+\alpha}{\alpha}^{-1} \sum\limits_{k=0}^n \binom{n-k+\alpha}{\alpha} a_k z^k,\quad \alpha \geqslant 0,
\end{align*}
where the binomial coefficient $\binom{m+\alpha}{\alpha}$ is given the following interpretation: 
\begin{align*}
\binom{m+\alpha}{\alpha} = \frac{\Gamma(m+\alpha+1)}{\Gamma(\alpha+1) \Gamma(m+1)}, 
\end{align*} 
with $\Gamma$ denoting the usual gamma function. 
Note that the generalized $n^{\it th}$-Ces{\`a}ro mean $\sigma_n^{\alpha}[f]$ corresponds to the $n^{\it th}$-partial sum $s_n[f]$ when $\alpha =0$ and corresponds to the the $n^{\it th}$-Ces{\`a}ro sum $\sigma_n[f]$ when $\alpha=1.$ 
For $\alpha\geqslant 0,$ the Taylor series of a function $f$ is said to be $(C,\alpha)$-summable to $f$ if the sequence $\{\sigma_n^{\alpha}[f]\}$ converges to $f$ in the associated normed function space. It is well known that $(C,\alpha)$-summability implies $(C,\beta)$-summability if $\alpha \leqslant \beta,$ see \cite[Theorem 43]{Hardy}. 
In this article, our primary objective is to investigate whether the sequence of generalized Cesàro means $\{\sigma_n^{\alpha}[f]\}$ converges to the corresponding function $f$ that belongs to the higher order weighted Dirichlet space $\mathcal H_{\mu,m}$ as defined below. 
These spaces serve as a crucial framework for modeling a specific sub-class within the broader category of $m$-isometries. The motivation behind introducing these spaces is rooted in Agler's exploration of $m$-isometries
see \cite{AglerStan1}, \cite{AglerStan2}, \cite{AglerStan3}. 
In order to find a suitable model for cyclic $m$-isometries, Rydhe delved into the study of higher order weighted Dirichlet type space, see \cite{Rydhe}.  Following Rydhe's framework, we consider, for a measure $\mu$ in $\mathcal M_{+}(\mathbb T)$, $f\in \mathcal O(\mathbb D)$ and for a positive integer $m,$ the concept of  {\it weighted Dirichlet-type integral of $f$ of order $m$}, $D_{\mu,m}(f),$ defined by 
\begin{align*}
D_{\mu,m}(f):= \frac{1}{m!(m-1)!}\displaystyle\int_{\mathbb D}\big|f^{(m)}(z)\big|^2 P_{\!\mu}(z)(1-|z|^2)^{m-1}dA(z).
\end{align*} 
Here $dA(z)$ denotes the normalized Lebesgue measure on the unit disc $\mathbb D,$ $f^{(m)}(z)$ represents the $m^{\it th}$-order derivative of $f$ at $z,$ and $P_{\!\mu}(z)$ is the Poisson integral of the measure $\mu,$ that is,
\begin{align*}
P_{\!\mu}(z):=\int_{\mathbb T} \frac{1-|z|^2}{|z-\zeta|^2} d\mu(\zeta),\,\,\,z\in \mathbb D.
\end{align*}  
 When dealing with Dirac delta measure $\delta_{\lambda}$, representing a point measure at $\lambda \in \mathbb T$, we adopt a simpler notation $D_{\lambda,m}(\cdot)$ in place of $D_{\delta_{\lambda},m}(\cdot)$. 
 For a measure $\mu\in \mathcal M_{+}(\mathbb T)/\{0\}$ and for each $m\in\mathbb N,$ 
 we consider the semi-inner product space $\mathcal H_{\mu, m}$ given by
$$\mathcal H_{\mu, m}:=\big\{f\in \mathcal O(\mathbb D): D_{\mu,m}(f)< \infty\big\},$$
associated to the semi-norm $\sqrt{D_{\mu,m}(\cdot)}.$ 
In case $\mu$ is $\delta_{\lambda}$ for some $\lambda\in\mathbb T,$ we will use a simpler notation $\mathcal H_{\lambda,m}$ in place of $\mathcal H_{\delta_{\lambda},m}$ and we refer it as a {\it local Dirichlet space of order $m$ at $\lambda$}. If $\mu = 0,$ we set $\mathcal H_{\mu,m}= H^2,$ Hardy space on $\mathbb D,$ for every $m\in \mathbb N.$  
If $\mu=\sigma,$ by a straightforward computation, it follows that for a holomorphic function $f=\sum_{k=0}^\infty a_k z^k$ in $\mathcal O(\mathbb D),$ we have 
\begin{equation*}
D_{\sigma, m}(f)=\sum_{k=m}^\infty \binom{k}{m}|a_k|^2, \,\,\,\quad m\in\mathbb Z_{\geqslant 0},
\end{equation*}
where $\binom{k}{m}:=\frac{k!}{m!(k-m)!}$ for any $k\geqslant m.$ Using this, it can be easily verified that $\mathcal H_{\sigma, m}$ coincides with the space $\mathcal D_{m}$, studied in \cite{TaylorDA}, where
$$\mathcal D_m=\Big\{\sum_{k=0}^\infty a_k z^k: \sum_{k=0}^\infty |a_k|^2 (k+1)^m < \infty\Big\}.$$

 The reader is referred to \cite{GGR}, \cite{GGR2}, \cite{L-R} and \cite{Rydhe} for several properties of the spaces $\mathcal H_{\mu,m}$ for an arbitrary non-negative measure $\mu$ and a positive integer $m.$ 
 Note that when $m=1,$ the weighted Dirichlet-type space $\mathcal H_{\mu,1}$ coincides with the weighted Dirichlet type space $D(\mu)$ as introduced in Richter \cite{R}. 
 Theorem \ref{mainthm1} and Theorem \ref{mainthm2} are the main results of this article. These results in the case of $m=1$ are well known, see \cite{MasRans, CesaroM}.
 \begin{theorem}\label{mainthm1}
 Let $\mu\in\mathcal M_+(\mathbb T)$ and $m\in \mathbb N.$ If $\alpha > \frac{1}{2}$ then there exists a constant $\kappa >0$ such that
 \begin{align*}
 D_{\mu,m}(\sigma_n^{\alpha}[f]) \leqslant \kappa D_{\mu,m}(f),\,\,~~~\mbox{for every~}n\in \mathbb N,\ f\in \mathcal H_{\mu,m}.
 \end{align*}
Moreover, for every $f\in \mathcal H_{\mu,m},$ $D_{\mu,m}\big(\sigma_n^{\alpha}[f]-f\big) \rightarrow 0 $ as $n\to \infty.$
 \end{theorem}

It is known that the space $\mathcal H_{\mu,m}$ is contained in the Hardy space $H^2$ (see \cite[Corollary 2.5]{GGR}). This allows us to define a norm $\|\cdot\|_{\mu,m}$ on $\mathcal H_{\mu,m}$ given by 
\[\|f\|_{\mu,m}^2=\|f\|^2_{H^2}+D_{\mu,m}(f),~f\in \mathcal H_{\mu,m},\] 
where $\|f\|_{H^2}$ denotes the norm of $f$ in $H^2.$
As a consequence of Theorem \ref{mainthm1}, we obtain the following corollary.
\begin{cor}\label{maincor} Let $\mu\in\mathcal M_+(\mathbb T)$ and $m\in 
\mathbb N.$ If $\alpha > \frac{1}{2}$, then $\|\sigma_n^{\alpha}[f]-f\|_{\mu, m}\to 0$ as $n\to \infty$ for any  $f\in \mathcal H_{\mu,m}$.
  \end{cor}
 
 Continuing our investigation into the convergence behavior of the sequence of generalized Cesàro means $\{\sigma_n^{\alpha}[f]\}$ to the respective function $f$ when $\alpha\leqslant \frac{1}{2}$ within the higher order weighted Dirichlet-type space $\mathcal H_{\mu,m}$, we bring attention to the following theorem.

 \begin{theorem}\label{mainthm2}
 Let $m\in\mathbb N$ and $\lambda\in \mathbb T$ and $\alpha=\frac{1}{2}.$ There exists a function $f\in \mathcal H_{\lambda,m}$ such that $D_{\lambda,m}\big(\sigma_n^{\alpha}[f]-f\big) \nrightarrow 0$ as $n\rightarrow\infty$.
 \end{theorem}

We will use the techniques of the Hadamard multiplication operators of $\mathcal H_{\mu,1}$ as developed in \cite{MasRans}, in order to prove our results.  In Section \ref{Hadamard Product}, first we extend the theory of Hadamard multiplication operators of $\mathcal H_{\mu,1}$ to higher order weighted Dirichlet type spaces $\mathcal H_{\mu,m}$ and then we proceed to establish the main results.

\section{Hadamard Multiplication Operators on $\mathcal H_{\mu,m}$}\label{Hadamard Product} 
For two formal power series $f(z)= \sum_{j=0}^{\infty}a_jz^j$ and $g(z)= \sum_{j=0}^{\infty}b_jz^j,$ the Hadamard product $f * g$ of $f$ and $g$ is defined by a formal power series, given by the formula 
\begin{align*}
(f * g)(z):= \sum\limits_{j=0}^{\infty} a_jb_jz^j.
\end{align*}
In view of the Cauchy-Hadamard formula for the radius of convergence of a power series, it is straightforward to verify that $f * g$ is in $\mathcal O(\mathbb D)$ whenever $f$ and $g$ belong to $\mathcal O(\mathbb D).$ For $\mu\in\mathcal M_{+}(\mathbb T)$ and $m\in\mathbb N,$ a function $c\in \mathcal O(\mathbb D)$ is said to be a Hadamard multiplier of $\mathcal H_{\mu,m}$ if $(c * f)\in \mathcal H_{\mu,m}$ for every $f\in \mathcal H_{\mu,m}.$ The following proposition characterizes the Hadamard multiplier of $\mathcal H_{\mu,m}.$
Before delving into the proposition, for $\mu\in\mathcal M_+(\mathbb T)\setminus \{0\},$ we introduce a linear space denoted as $\hat{\mathcal H}_{\mu,m}$, associated to the linear space $\mathcal H_{\mu,m},$ for the sake of computational simplicity. 
\begin{align}\label{seminorm Hilbert space}
\hat{\mathcal H}_{\mu,m} := \Big\{ f\in \mathcal O(\mathbb D) : f(z)= \sum\limits_{j=m}^{\infty}a_jz^j,~~a_j\in\mathbb C,~~ D_{\mu,m}(f) < \infty\Big\}.
\end{align} 
When $\mu=0,$ we define $\hat{\mathcal H}_{\mu,m}$ to be the Hardy space $H^2.$
For a non-zero $\mu$ in $\mathcal M_+(\mathbb T),$ since $\sqrt{D_{\mu,m}(\cdot)}$ is a semi-norm on the linear space $\mathcal H_{\mu,m}$ and $D_{\mu,m}(f)=0$ if and only if $f$ is a polynomial of degree at most $(m-1),$ it follows that $\sqrt{D_{\mu,m}(\cdot)}$ induces a norm on $\hat{\mathcal H}_{\mu,m}.$
It is straightforward to verify that $\hat{\mathcal H}_{\mu,m}$ turns out to be a Hilbert space with respect to this norm.
Furthermore, as established in the following lemma, $\hat{\mathcal H}_{\mu,m}$ turns is a reproducing kernel Hilbert space on $\mathbb D.$

\begin{lemma}\label{RKHS Lemma}
Let $\mu\in\mathcal M_{+}(\mathbb T)\setminus \{0\}$ and $m\in\mathbb N.$ The linear space $\hat{\mathcal H}_{\mu,m},$ equipped with the norm  $\sqrt{D_{\mu,m}(\cdot)},$ is a reproducing kernel Hilbert space on $\mathbb D.$
\end{lemma}
\begin{proof}
Let $f\in \hat{\mathcal H}_{\mu,m}$ and $f(z)=\sum_{j=m}^{\infty}a_jz^j,$  $z\in\mathbb D.$ Then we have
\begin{align*}
f^{(m)}(z) &= \sum\limits_{j=m}^{\infty} \frac{\Gamma(j+1)}{\Gamma(j-m+1)}a_jz^{j-m},~~~~~~~z\in\mathbb D.
\end{align*}
In view of \cite[Eq (2.6), p. 454]{GGR}, we obtain that
\begin{align*}
D_{\mu,m}(f)& \geqslant \frac{\mu(\mathbb T)}{4m!(m-1)!} \int_{\mathbb D}|f^{(m)}|^2(z)(1-|z|^2)^m dA(z)\\
&=\frac{\mu(\mathbb T)}{4m!(m-1)!} \sum\limits_{j=m}^{\infty} |a_j|^2 \frac{(\Gamma(j+1))^2}{(\Gamma(j-m+1))^2} \int_{\mathbb D} |z|^{2(j-m)}(1-|z|^2)^m dA(z)\\
 &= \frac{\mu(\mathbb T)}{4m!(m-1)!}\sum\limits_{j=m}^{\infty} |a_j|^2 \frac{(\Gamma(j+1))^2}{(\Gamma(j-m+1))^2} \frac{\Gamma(m+1) \Gamma(j-m+1)}{\Gamma(j+2)} \\
& =\frac{\mu(\mathbb T)}{4(m-1)!}\sum\limits_{j=m}^{\infty} \frac{|a_j|^2}{(j+1)} \frac{\Gamma(j+1)}{\Gamma(j-m+1)} 
\end{align*}
Using Cauchy-Schwarz inequality, we have for each $w\in\mathbb D,$ 
\begin{align*}
|f(w)|^2 = |\sum_{j=m}^{\infty}a_jw^j|^2 &\leqslant \Big(\sum\limits_{j=m}^{\infty}  \frac{|a_j|^2}{(j+1)} \frac{\Gamma(j+1)}{\Gamma(j-m+1)}\Big)   \Big(\sum\limits_{j=m}^{\infty} \frac{(j+1)\Gamma(j-m+1)}{\Gamma(j+1)} |w|^{2j}\Big)\\
&= \frac{4(m-1)!}{\mu(\mathbb T)}\Big(\sum\limits_{j=m}^{\infty} \frac{(j+1)\Gamma(j-m+1)}{\Gamma(j+1)} |w|^{2j}\Big) D_{\mu,m}(f).
\end{align*}
This shows that the evaluation at each point $w\in \mathbb D$ is a bounded linear functional on $\hat{\mathcal H}_{\mu,m}$ and hence $\hat{\mathcal H}_{\mu,m}$ is a reproducing kernel Hilbert space on $\mathbb D.$ This completes the proof.
\end{proof}

Now we establish the result that gives a characterization of the Hadamard multiplier of $\mathcal H_{\mu,m}.$
\begin{prop}\label{Hadamard mult and seminorm inequality}
Let $\mu\in\mathcal M_{+}(\mathbb T)$ and $m\in\mathbb N.$ A function $c\in \mathcal O(\mathbb D)$ is a Hadamard multiplier of $\mathcal H_{\mu,m}$ if and only if there exists a constant $\kappa > 0$ such that
\begin{align*}
D_{\mu,m}(c*f) \leqslant \kappa D_{\mu,m}(f),~~~~~~~\mbox{for all~~~~}f\in \mathcal H_{\mu,m}.
\end{align*}
\end{prop}
\begin{proof}
Suppose that a function $c\in \mathcal O(\mathbb D)$ is a Hadamard multiplier of $\mathcal H_{\mu,m}.$ Thus $(c*f)\in \hat{\mathcal H}_{\mu,m}$ for every $f\in \hat{\mathcal H}_{\mu,m}.$ 
Consider the linear transformation $H_c: \hat{\mathcal H}_{\mu,m} \to \hat{\mathcal H}_{\mu,m}$ given by $H_c(f)=c*f,$ for $f\in \hat{\mathcal H}_{\mu,m}.$ 
Since, by Lemma \ref{RKHS Lemma}, $\hat{\mathcal H}_{\mu,m}$ is a reproducing kernel Hilbert space consisting of analytic functions on $\mathbb D$, it follows that the linear functional $ L_j: f\to \frac{f^{(j)}(0)}{j!}$ on $\hat{\mathcal H}_{\mu,m}$ is continuous for every $j\geqslant m$ (see \cite[Lemma 4.1]{CS-84}). This gives us that the graph of $H_c$ is closed, and hence by closed graph theorem, $H_c$ is bounded. That is, $D_{\mu,m}(c*f) \leqslant \|H_c\| D_{\mu,m}(f)$ for every $f\in \hat{\mathcal H}_{\mu,m}.$ For an arbitrary $f\in \mathcal H_{\mu,m},$ we write
\begin{align*}
f= p+g, 
\end{align*}
where $p$ is a polynomial of degree at most $m-1$ and $g\in \hat{\mathcal H}_{\mu,m}.$
Hence we obtain 
\begin{align*}
D_{\mu,m}(c*f)=  D_{\mu,m}(c*g) \leqslant \|H_c\| D_{\mu,m}(g) = \|H_c\| D_{\mu,m}(f). 
\end{align*}
The converse part is trivial.
\end{proof}
For any $c(z)=\sum_{j=0}^{\infty}c_jz^j$ in $\mathcal O(\mathbb D)$ and for any $m\in\mathbb N,$ we consider the infinite matrix $T_c(m)$ defined by
\begin{align*}
T_c(m):= 
\begin{pmatrix}
c_{m} & c_{m+1}-c_{m} & c_{m+2}-c_{m+1} & c_{m+3}-c_{m+2} & \cdots \\
0 & c_{m+1}  & c_{m+2}-c_{m+1} & c_{m+3}-c_{m+2} & \cdots \\
0 & 0 & c_{m+2}  & c_{m+3}-c_{m+2} & \cdots \\
0 & 0 & 0 & c_{m+3} & \cdots \\
\vdots & \vdots & \vdots & \vdots & \ddots
\end{pmatrix}.
\end{align*} 
Let $W_m$ be the linear space in $\mathcal O(\mathbb D)$ given by $W_m= span \{z^j: j\geqslant m-1\}.$ The matrix $T_c(m)$ induces a linear transformation $A_c(m)$ on the  linear space $W_m$, given by 
\begin{align*}
A_c(m) \Big(\sum\limits_{j=m-1}^{n} b_j z^j\Big):=  \sum\limits_{j=m-1}^{n} \Big(c_{j+1}b_j + \sum\limits_{k=j+1}^{n}(c_{k+1}-c_k)b_k \Big) z^j.
\end{align*}
Let $\sigma$ be the normalized Lebesgue measure on $\mathbb T.$ Consider the Hilbert space $\hat{\mathcal H}_{\sigma,m-1}$ associated with the norm $\sqrt{D_{\sigma,m-1}(\cdot)},$ as defined in \eqref{seminorm Hilbert space}. 
Suppose $A_c(m)$ induces a bounded operator on the Hilbert space $\hat{\mathcal H}_{\sigma,m-1}$. Since $\{z^j:j\geqslant m-1\}$ is an orthogonal basis for $\hat{\mathcal H}_{\sigma,m-1},$ it easily follows that 
\begin{align}\label{eqn5}
 D_{\sigma,m-1} \Big(A_c(m)^*(z^{m-1})\Big) =|c_m|^2 + \sum_{j=m}^{\infty}|c_{j+1}-c_{j}|^2 \binom{j}{m-1}^{-1}.
\end{align}
Let $b(z)=\sum_{j=0}^{\infty} b_j z^j \in \mathcal H_{\sigma,m-1}.$ Then we have $D_{\sigma, m-1}(b)=\sum _{j=m-1}^\infty |b_j|^2 \binom{j}{m-1} < \infty$. This together with \eqref{eqn5} shows that the infinite series  $\sum_{k=j+1}^{\infty}(c_{k+1}-c_k)b_k$ converges absolutely for every $j\geqslant m$. 
We indeed have that
\begin{align}\label{A_c(m)}
A_c(m) \Big(\sum\limits_{j=m-1}^{\infty} b_j z^j\Big)=  \sum\limits_{j=m-1}^{\infty} \Big(c_{j+1}b_j + \sum\limits_{k=j+1}^{\infty}(c_{k+1}-c_k)b_k \Big) z^j.
\end{align}

The following theorem tells us that $A_c(m)$ being a bounded linear operator is equivalent to $c$ being a Hadamard multiplier of $\mathcal H_{\mu,m}$ for every $\mu\in\mathcal M_+(\mathbb T).$ 
\begin{theorem}\label{Hadamard}
Let $m\in\mathbb N,$ $c\in \mathcal O(\mathbb D).$ The following statements are equivalent.
\begin{itemize}
\item[(i)] $c$ is a Hadamard multiplier of $\mathcal H_{\mu,m}$ for every $\mu\in\mathcal M_+(\mathbb T).$
\item[(ii)] $c$ is a Hadamard multiplier of $\mathcal H_{\lambda,m}$ for some $\lambda\in\mathbb T.$
\item[(iii)] The transformation $A_c(m)$ defines a bounded operator on $\hat{\mathcal H}_{\sigma,m-1}.$ 
 \end{itemize}
Moreover, in this case, for every $\mu\in\mathcal M_+(\mathbb T),$ we have 
\begin{align*}
D_{\mu,m}(c*f) \leqslant \|A_c(m)\|^2 D_{\mu,m}(f),~~~~f\in \mathcal H_{\mu,m}.
\end{align*}

\end{theorem}

\begin{proof}
$(iii) \implies (i) :$ Assume that the linear transformation $A_c(m)$ defines a bounded operator on the Hilbert space $\hat{\mathcal H}_{\sigma,m-1}.$ Fix a $\lambda\in \mathbb T$. 
 Let $f\in \mathcal  \mathcal H_{\lambda, m}$ and the associated power series representation of $f$ be given by 
$f(z)= \sum_{j=0}^{\infty}a_jz^j.
$
By \cite[Theorem 1.1]{GGR2}, $f(z)=a + (z-\lambda)L_{\lambda}[f](z)$, $z\in \mathbb D$, for some $a\in\mathbb C$ and $L_{\lambda}[f]\in\mathcal H_{\sigma,m-1}$ with  $D_{\lambda,m}(f)=D_{\sigma, m-1}(L_{\lambda}[f])$.  Writing $L_{\lambda}[f](z)= \sum_{j=0}^{\infty}b_jz^j$, we obtain the relations
\begin{align}\label{eqnHadmard1}
a_0=a-b_0\lambda,~~~a_k= b_{k-1}-b_k\lambda, ~~~ k \geqslant 1.
\end{align} 
Note that 
\begin{align*}
D_{\lambda,m}(f)= D_{\sigma,m-1}(L_{\lambda}[f]) = \sum_{j=m-1}^{\infty} |b_j|^2 \binom{j}{m-1} = D_{\sigma,m-1} \Big(\sum_{j=m-1}^{\infty} b_j z^j\Big).
\end{align*}
As $A_c(m)$ defines a bounded operator on $\hat{\mathcal H}_{\sigma,m-1}$ and $|\lambda|=1,$ the series $\sum_{k=j+1}^{\infty}(c_{k+1}-c_k)b_k \lambda^{k}$ converges absolutely for every $j\geqslant m$ (see the discussion preceding this theorem)  and it makes sense to consider the formal power series $g$ given by 
\begin{align*}
g(z)&= \sum\limits_{j=0}^{\infty} \Big(c_{j+1}b_j + \sum\limits_{k=j+1}^{\infty}(c_{k+1}-c_k)b_k \lambda^{k-j}\Big) z^j\\
&= \sum\limits_{j=0}^{\infty} \Big(c_{j+1}b_j\lambda^j + \sum\limits_{k=j+1}^{\infty}(c_{k+1}-c_k)b_k \lambda^{k}\Big) (\bar{\lambda}z)^j.
\end{align*}
Then by \eqref{eqnHadmard1}, it follows that $(c*f)(z)= A + (z-\lambda)g(z)$ for every $z\in \mathbb D$ for some $A\in\mathbb C$. Note that 
$$D_{\sigma, m-1}\big(\sum_{j=m-1}^{\infty} b_j \lambda^jz^j\big)=D_{\sigma, m-1}(\sum_{j=m-1}^{\infty} b_j z^j)=D_{\lambda, m}(f),$$ and 
\begin{align*}
 D_{\sigma,m-1}(g) &=  \sum_{j=m-1}^{\infty} \big| c_{j+1}b_j\lambda^j + \sum\limits_{k=j+1}^{\infty}(c_{k+1}-c_k)b_k \lambda^{k}\big|^2 \binom{j}{m-1}\\
&=  D_{\sigma,m-1}\Big( A_c(m)\big(\sum_{j=m-1}^{\infty} b_j \lambda^jz^j\big)\Big) < \infty.
\end{align*} 
Thus applying \cite[Theorem 1.1]{GGR2} a second time, we obtain $c*f\in \mathcal H_{\lambda, m}$, and 
\begin{eqnarray}\label{Rel-Had(c,f)-and-g}
D_{\lambda, m}(c*f)=D_{\sigma,m-1}(g).
\end{eqnarray}
Hence we obtain that 
\begin{equation}\label{equality norm}
\sup \Big\{ D_{\lambda,m}(c*f) : D_{\lambda,m}(f)=1 \Big\} = \|A_c(m)\|^2.
\end{equation}
This gives us that $D_{\lambda,m}(c*f) \leqslant \|A_c(m)\|^2 D_{\lambda,m}(f)$ for every $f\in H_{\lambda,m}.$ Since $\lambda\in \mathbb T$ was arbitrary, for any  $f\in \mathcal H_{\mu,m}$, it follows that  
\begin{align*}
D_{\mu,m}(c*f)& =\int_{\mathbb T} D_{\lambda,m}(c*f) d\mu(\lambda) \\ &\leqslant  \|A_c(m)\|^2 \int_{\mathbb T} D_{\lambda,m}(f) d\mu(\lambda)\\ &= \|A_c(m)\|^2 D_{\mu,m}(f).
\end{align*}
Hence we obtain that $c$ is a Hadamard multiplier of $\mathcal H_{\mu,m}$ for every $\mu\in \mathcal M_+(\mathbb T).$ 

$(i) \implies (ii)$ is obvious.

$(ii) \implies (iii) :$ Assume that $c$ is a Hadamard multiplier of $\mathcal H_{\lambda,m}$ for some $\lambda\in\mathbb T.$ Following Proposition \ref{Hadamard mult and seminorm inequality}, we obtain that the linear transformation $H_c$ on the Hilbert space $\hat{\mathcal H}_{\lambda,m}$ given by $H_c(f)=c*f,$ for $f\in \hat{\mathcal H}_{\lambda,m}$ is bounded. Moreover, there exists a constant $\kappa =\|H_c\|^2$ such that
\begin{align}\label{norm of Hc}
D_{\lambda,m}(c*f) \leqslant \kappa D_{\lambda,m}(f),~~~~~~~\mbox{for all~~~~}f\in \mathcal H_{\lambda,m}.
\end{align}

Note that, by \cite{GGR2},  every $f\in \mathcal H_{\lambda,m}$ can be uniquely written as $f(z)=a + (z-\lambda)L_{\lambda}[f](z)$, where $a\in \mathbb C$ and $D_{\lambda,m}(f)=D_{\sigma,m-1}(L_{\lambda}[f]).$ As $c$ is a Hadamard multiplier of $\mathcal H_{\lambda,m}$,  the inequality in \eqref{norm of Hc} can be rephrased as
\begin{align}\label{Revised inequality}
D_{\sigma,m-1}(L_{\lambda}[c*f]) \leqslant \kappa D_{\sigma,m-1}(L_{\lambda}[f]),~~~~~~~\mbox{for all~~~~}f\in \mathcal H_{\lambda,m}.
\end{align}
In view of the higher order local Douglas formula \cite{GGR2}, we have that for any $g\in \mathcal H_{\sigma,m-1},$ the function $(z-\lambda)g \in \mathcal H_{\lambda,m}$ and consequently $D_{\sigma,m-1}(L_{\lambda}[c*(z-\lambda)g]) < \infty$. Since for any $f\in\mathcal H_{\sigma,m-1},$ the function $f-s_{m-2}[f]\in \hat{\mathcal H}_{\sigma,m-1}$ whenever $m\geqslant 2.$ Now consider the transformation $\hat{H}_c(m)$ from the Hilbert space $\hat{\mathcal H}_{\sigma,m-1}$ into itself defined by
\begin{align*}
\hat{H}_{c}(m)(g):= \begin{cases}
L_{\lambda}[c*(z-\lambda)g] & \mbox{if~~} m=1,\\
L_{\lambda}[c*(z-\lambda)g] - s_{m-2}[L_{\lambda}[c*(z-\lambda)g]] & ~~~~\mbox{if~~} m \geqslant 2,
\end{cases}
~~~~\,\,\,\,\,\,\,g\in \hat{\mathcal H}_{\sigma,m-1}.
\end{align*}
In view of the inequality \eqref{Revised inequality}, we obtain that
\begin{align*}
D_{\sigma,m-1}\Big(\hat{H}_{c}(m)(g)\Big) &= D_{\sigma,m-1}(L_{\lambda}[c*(z-\lambda)g])\\
&\leqslant \kappa D_{\sigma,m-1}(L_{\lambda}[(z-\lambda)g])\\
&= \kappa D_{\sigma,m-1}(g),
\end{align*}
for every $g\in \hat{\mathcal H}_{\sigma,m-1}.$ Let's compute the matrix representation of $\hat{H}_{c}(m)$ with respect to the orthogonal basis $\{(\bar{\lambda}z)^j : j\geqslant m-1 \}$ of $\hat{\mathcal H}_{\sigma,m-1}.$ Note that for any $j\geqslant (m-1)$
\begin{align*}
L_{\lambda}\Big[ c* (z-\lambda)(\bar{\lambda}z)^j\Big]&= L_{\lambda}\Big[c_{j+1}\bar{\lambda}^j z^{j+1}- c_j\bar{\lambda}^{j-1}z^j\Big] \\
& = c_{j+1}\bar{\lambda}^j \frac{z^{j+1}-{\lambda}^{j+1}}{z-\lambda} - c_{j}\bar{\lambda}^{j-1} \frac{z^{j}-{\lambda}^{j}}{z-\lambda}\\
&=  (c_{j+1}-c_j)\sum_{k=0}^{j-1}(\bar{\lambda}z)^k + c_{j+1}(\bar{\lambda}z)^j
\end{align*}
Thus it follows that for any $j\geqslant (m-1),$ we have
\begin{align*}
\hat{H}_{c}(m)\Big((\bar{\lambda}z)^j\Big) = (c_{j+1}-c_j)\sum_{k=m-1}^{j-1}(\bar{\lambda}z)^k + c_{j+1}(\bar{\lambda}z)^j.
\end{align*}
Hence the matrix representation of the operator $\hat{H}_{c}(m)$ with respect to the orthogonal basis $\{(\bar{\lambda}z)^j : j\geqslant m-1 \}$ of $\hat{\mathcal H}_{\sigma,m-1}$ coincides with the matrix $T_c(m).$ Let $V$ be the unitary operator  on $\hat{\mathcal H}_{\sigma,m-1}$ defined by 
\begin{align*}
V\Big(\sum_{j=m-1}^{\infty}b_jz^j\Big)= \sum_{j=m-1}^{\infty}b_j {(\bar{\lambda} z)}^j
\end{align*}
Now it is straightforward to verify that $V^{-1}\hat{H}_{c}(m)V (z^j)= A_c(m)(z^j)$ for every $j\geqslant (m-1).$ Hence $A_c(m)$ must define a bounded operator on $\hat{\mathcal H}_{\sigma,m-1}$ and $\|A_c(m)\|= \|\hat{H}_{c}(m)\|.$
\end{proof}

\begin{rem}
It is straightforward to verify that any function $c\in \mathcal O(\mathbb D)$ with  Taylor coefficients $\{c_n\}$ is a Hadamard multiplier for $\mathcal H_{\sigma,m}$ if and only if $\{c_n\}$ is bounded. On the other hand, if we define $c_n=1$ if $n$ is odd and $c_n=0$ if $n$ is even, then the linear transformation $A_c(m)$ will be unbounded on $\hat{\mathcal H}_{\sigma, m-1}$, as noted in \cite[pg 52]{CesaroM} in the case $m=1.$
To see this for general $m$, note that for any even $n\geqslant m,$
\begin{eqnarray*}
\big\|A_c(m)(z^n)\big\|^2_{\sigma, m-1}= \sum_{k=m-1}^{n}{k\choose m-1}={n+1\choose m}.
\end{eqnarray*}
Therefore $\|A_c(m)\|\geqslant {n\choose m-1}^{-\frac{1}{2}}{n+1\choose m}^{\frac{1}{2}}\sim \sqrt{\frac{(n+1)^{m}}{n^{m-1}}}$ for all even  $n\geqslant m.$ Hence $A_c(m)$ is unbounded on $\hat{\mathcal H}_{\sigma, m-1}$.
\end{rem}
\begin{rem}\label{rem norm}
If $c$ is a Hadamard Multiplier of $\mathcal H_{\lambda,m}$ for some $\lambda\in\mathbb T$, then by Proposition \ref{Hadamard mult and seminorm inequality} the operator $H_c:\hat{\mathcal H}_{\lambda,m}\to \hat{\mathcal H}_{\lambda,m}$, defined by $H_c(f)=c*f$, is bounded. In view of \eqref{equality norm}, it follows that, in this case, $$\|H_c:\hat{\mathcal H}_{\lambda,m}\to \hat{\mathcal H}_{\lambda,m}\|=\|A_c(m)\|.$$
\end{rem}

\section{Generalized Ces{\`a}ro  summability in higher order Dirichlet spaces}
In this section, we provide proofs of Theorem \ref{mainthm1}  and Theorem \ref{mainthm2}. We shall need
the following property of the Gamma function  (see \cite[p. 257 (6.1.46)]{MS-64}):
\begin{equation}\label{stirling}
\lim_{k\to \infty} k^{b-a} \frac{\Gamma(k+a)}{\Gamma(k+b)}=1,~a,b\in\mathbb C.
\end{equation}
We start with the following lemma.
\begin{lemma}\label{apprx partial sum}
Let $\alpha\in \big(\frac{1}{2},1\big).$
Then there exists a positive constant $C$ such that
$$\sum_{j=0}^n\frac{\Gamma(j+\alpha)^2}{\Gamma(j+1)^2}\leqslant C (n+1)^{2\alpha-1}~~\mbox{for all}~n\geqslant 0.$$ 
\end{lemma}
\begin{proof}
By \eqref{stirling}, there exists a constant $C_1>0$ such that $\frac{\Gamma(k+\alpha)}{\Gamma(k+1)}\leqslant C_1 (k+1)^{\alpha-1}$ for all $k\in \mathbb Z_{\geqslant 0}$.
Thus we have
\begin{align*}
\sum_{j=0}^n\frac{\Gamma(j+\alpha)^2}{\Gamma(j+1)^2}&\leqslant C_1^2\sum_{j=0}^n (j+1)^{2\alpha-2}\\
&\leqslant C_1^2\sum_{j=0}^n \int_j^{j+1} t^{2\alpha-2}~dt\\
&=C_1^2\int_{j=0}^{n+1} t^{2\alpha-2}~dt\\
&=\frac{C_1^2}{2\alpha-1}(n+1)^{2\alpha-1}.
\end{align*}
Choosing $C=\frac{C_1^2}{2\alpha-1}$ completes the proof of the lemma.
\end{proof}

\begin{theorem}\label{uniformly bounded M-n-alpha}
Let $\alpha\in \big(\frac{1}{2},1\big).$ For each $n\in\mathbb N,$ let $h_n(z)=\binom{n+\alpha}{\alpha}^{-1} \sum\limits_{k=0}^n \binom{n-k+\alpha}{\alpha} z^k$. Then the family of linear transformations
\begin{equation*}
A_{h_n}(m):\hat{\mathcal H}_{\sigma, m-1}\to \hat{\mathcal H}_{\sigma, m-1},~ n\in\mathbb N,
\end{equation*}
is uniformly bounded in operator norm.
\end{theorem}
\begin{proof}
Note that $\big\{ {j\choose m-1}^{-\frac{1}{2}} z^{j}\big\}_{j\geqslant m-1}$ forms an orthonormal basis of $\hat{\mathcal H}_{\sigma, m-1}.$ Fix $n\in\mathbb N.$ 
With respect to this orthonormal basis, let $(\!(a_{i,j} )\!)_{i,j=m-1}^\infty$ denote the matrix of  $A_{h_n}(m).$
From \eqref{A_c(m)}, it is easy to see that  
\begin{equation*}\label{the matrix aij}
a_{i,j}=\begin{cases}
c_{j+1},& \mbox{if~}  i=j\\
\sqrt{\frac{{i \choose m-1}}{{j \choose m-1}}}(c_{j+1}-c_{j}),& \mbox{if~} i+1\leqslant j\\
0,& \mbox{otherwise},
\end{cases}
\end{equation*}
where $c_j=\binom{n+\alpha}{\alpha}^{-1} \binom{n-j+\alpha}{\alpha}$ for $j\leqslant n$ and $c_j=0$ for $j>n.$
Note that all the entries of the matrix $(\!(a_{i,j} )\!)_{i,j=m-1}^\infty$ are zero except finitely many. 
Thus we have 
\begin{align*}
\|A_{h_n}(m)\|
&=\|(\!( a_{i,j} )\!)_{i,j=m-1}^n\| \nonumber\\
& \leqslant
\left \|\begin{pmatrix}
a_{m-1,m-1}& 0 & \hdots  & 0\\
0& a_{m,m}& \hdots & 0\\
\vdots &\vdots&\ddots &0\\
0&0&\hdots & a_{n,n}
\end{pmatrix}\right \|+ \left
\|\begin{pmatrix}
0& a_{m-1,m} & \hdots  & a_{m-1,n}\nonumber\\
0& 0& \hdots & a_{m,n}\\
\vdots &\vdots&\ddots &a_{n-1,n}\nonumber\\
0&0&\hdots & 0
\end{pmatrix}\right \| \nonumber\\
& \leqslant \max_{m-1\leqslant i\leqslant n} |a_{i,i} |+\sum_{j=m}^n \sum_{i=m-1}^{j} |a_{i,j}|^2\nonumber\\
& \leqslant 1 + \sum_{j=m}^n \sum_{i=m-1}^{j-1} |a_{i,j}|^2.\label{eqn2}
\end{align*}
Here for the second last inequality, we have used the fact that the operator norm of a matrix is bounded by its Hilbert- Schmidt norm.
It is easily verified that for $0\leqslant j\leqslant n$
\begin{align*}
c_{j+1}-c_j&=\frac{1}{{n+\alpha\choose \alpha}}\left({n-j-1+\alpha \choose \alpha}-{n-j+\alpha \choose \alpha}\right)\\
&=-\frac{\alpha~\Gamma(n-j+\alpha)}{{n+\alpha\choose \alpha}\Gamma(\alpha+1) \Gamma (n-j+1)}.
\end{align*}
Note that
\begin{align*}
\sum_{j=m}^n \sum_{i=m-1}^{j-1} |a_{i,j}|^2&
\leqslant \frac{\alpha^2}{(\Gamma(\alpha+1))^2{n+\alpha \choose \alpha}^2} \sum_{j=m}^n \sum_{i=m-1}^{j-1} \frac{{i\choose m-1}}{{j \choose m-1}} \frac{(\Gamma(n-j+\alpha))^2}{(\Gamma(n-j+1))^2}\nonumber\\
&
= \frac{\alpha^2}{(\Gamma(\alpha+1))^2{n+\alpha \choose \alpha}^2} \sum_{j=m}^n \frac{1}{{j \choose m-1}}\frac{(\Gamma(n-j+\alpha))^2}{(\Gamma(n-j+1))^2}\sum_{i=m-1}^{j-1}  {i\choose m-1}\nonumber\\&
=\frac{\alpha^2}{(\Gamma(\alpha+1))^2{n+\alpha \choose \alpha}^2} \sum_{j=m}^n \frac{1}{{j \choose m-1}}\frac{(\Gamma(n-j+\alpha))^2}{(\Gamma(n-j+1))^2} {j\choose m}\nonumber\\
&= \frac{\alpha^2}{m(\Gamma(\alpha+1))^2{n+\alpha \choose \alpha}^2} \sum_{j=m}^n (j-m+1)\frac{(\Gamma(n-j+\alpha))^2}{(\Gamma(n-j+1))^2}\nonumber\\&
\leqslant \frac{\alpha^2(n-m+1)}{m(\Gamma(\alpha+1))^2{n+\alpha \choose \alpha}^2}  \sum_{j=m}^n \frac{(\Gamma(n-j+\alpha))^2}{(\Gamma(n-j+1))^2}\nonumber\\
&=\frac{\alpha^2(n-m+1)}{m(\Gamma(\alpha+1))^2{n+\alpha \choose \alpha}^2}  \sum_{j=0}^{n-m}\frac{(\Gamma(j+\alpha))^2}{(\Gamma(j+1))^2}\\
& \leqslant C \frac{\alpha^2}{m(\Gamma(\alpha+1))^2{n+\alpha \choose \alpha}^2} (n-m+1)^{2\alpha}.
\end{align*}
Here, we have used a well-known binomial identity $\sum_{i=m-1}^{j-1}  {i\choose m-1}= {j\choose m},$ (see \cite[p. 46]{DW}), while the  last inequality follows from Lemma \ref{apprx partial sum}. 
 Also, by \eqref{stirling}, we get ${n+\alpha \choose \alpha}\sim (n+1)^{\alpha}$. 
 This completes the proof.
\end{proof}


The following lemma might be well-known to the experts. We provide a proof for the sake of completeness. 
\begin{lemma}\label{LemmaGMR}
Let $\{T_n\}_{n\geqslant 1}$ be a sequence of bounded linear operators on a reproducing kernel Hilbert space $\mathcal H$ of holomorphic functions on $\mathbb D$. Suppose that
\begin{enumerate}
\item[\rm (i)] $T_n$ is finite-rank for each $n\in\mathbb N$
\item[\rm (ii)] $T_nT_m(\mathcal H)\subseteq T_m(\mathcal H)$ for each $m,n\in \mathbb N$
\item[\rm (iii)] $T_n(f)(z)\to f(z)$ as $n\to \infty$ for all $f\in\mathcal H$ and for all $z\in \mathbb D.$ 
\end{enumerate}
 Then $T_n(f)\to f$ in norm as $n\to \infty$ for all $f\in\mathcal H$ if and only if $\sup_{n}\|T_n\| <\infty.$
\end{lemma}
\begin{proof}
Suppose  $T_n(f)\to f$ in norm as $n\to \infty$ for all $f\in\mathcal H$. Then, by the uniform boundedness principle, it follows that $\sup_{n}\|T_n\| <\infty.$ For the converse, assume that $\sup_{n}\|T_n\| <\infty.$ Let $K(z,w)$
denote the reproducing kernel of $\mathcal H$. Since $T_n(f)(z)\to f(z)$ as $n\to \infty$ for all $z\in \mathbb D$, by the reproducing property of $\mathcal H$, it follows that 
$\langle T_n(f),g\rangle \to \langle f, g\rangle$ for all $g$ of the form $\sum_{i=1}^\ell a_i K(\cdot, w_i)$, where $a_i\in\mathbb{C}, w_i\in\mathbb D, \ell\in\mathbb N.$ Since the set $\{\sum_{i=1}^\ell a_i K(\cdot, z_i): a_i\in\mathbb{C}, z_i\in\mathbb D, \ell\in\mathbb N\}$ is dense in $\mathcal H$, and $\sup_n\|T_n\|<\infty$, it follows that $T_n(f)\to f$ weakly as  $n\to \infty$ for all $f\in \mathcal H.$ The proof is now complete by \cite[Lemma 2.3]{GMR}.
\end{proof}

We now are ready to prove main theorem of this section. \\

\noindent \textbf{Proof of Theorem \ref{mainthm1}} 
In view of \cite[Theorem 43]{Hardy}, it is sufficient to consider $\frac{1}{2} <\alpha <1.$ 
Suppose $f\in\mathcal H_{\mu,m}.$
Note that $\sigma_n^{\alpha}[f](z)=(h_n*f)(z),$ where $h_n(z)=\binom{n+\alpha}{\alpha}^{-1} \sum_{k=0}^n \binom{n-k+\alpha}{\alpha} z^k$,  $n\in \mathbb N$.
By Theorem \ref{uniformly bounded M-n-alpha}, there exists a constant 
$\kappa>0$ such that $\|A_{h_n}(m)\| \leqslant \kappa$ for all $n\in\mathbb N.$ Using this along with Theorem 
\ref{Hadamard}, we get
\begin{eqnarray*}
D_{\mu,m}(\sigma_n^{\alpha}[f])=D_{\mu,m}(h_n*f) \leqslant \|A_{h_n}(m)\|^2 D_{\mu,m}(f)\leqslant \kappa^2 D_{\mu,m}(f).
\end{eqnarray*}
To prove the second part, let $T_n$ be the operator on $(\hat{\mathcal H}_{\mu,m},D_{\mu,m}(\cdot))$ defined by $T_n(f)=\sigma_n^{\alpha}[f].$ It is easy to see that $T_n$ is finite-rank and $T_nT_{\ell}(\hat{\mathcal H}_{\mu,m})\subseteq T_{\ell}(\hat{\mathcal H}_{\mu,m})$ for each $n,\ell$. Also, since for any $f\in\mathcal O(\mathbb D)$, $s_n[f](z)\to f(z)$ for each $z\in \mathbb D$, by 
\cite[Theorem 4.3]{Hardy}, we obtain that $T_n(f)(z)\to f(z)$ for each $z\in\mathbb D$. Further, by the first part of this Theorem, $\sup_n\|T_n\|<\infty.$ Hence by Lemma \ref{LemmaGMR}, it follows that $D_{\mu,m}(\sigma_n^\alpha[f]-f)\to 0$ as $n\to \infty$ for all $f\in \hat{\mathcal H}_{\mu,m}$. Now let $f=\sum_{i\geqslant 0} a_iz^i\in {\mathcal H}_{\mu,m}.$ Then $f_1:=\sum_{i\geqslant m} a_iz^i\in \hat{\mathcal H}_{\mu,m}$ and $D_{\mu,m}(\sigma_n^\alpha[f]-f)=D_{\mu,m}(\sigma_n^\alpha[f_1]-f_1)\to 0$ as $n\to\infty$. This completes the proof.
\vspace{.2cm}

\noindent \textbf{Proof of Corollary \ref{maincor}~}  
Let $\alpha> \frac{1}{2}.$  It is easy to see that $\|s_n[f]\to f\|_{H^2}\to 0$ for all $f\in H^2$. Thus by  \cite[Theorem 4.3]{Hardy}, $\|\sigma_n^{\alpha}[f]-f\|_{H^2}\to 0$ for all $f\in H^2$. This together with  Theorem \ref{mainthm1} completes the proof of this Corollary.


%
%
%
We proceed now to show that the statement of Theorem \ref{uniformly bounded M-n-alpha} does not hold true if $\alpha\leqslant 1/2.$ It is enough to disprove the statement for $\alpha=1/2.$

\begin{prop}\label{Propo3.4}
Suppose for each $n\in\mathbb N,$ $g_n(z)=\sum\limits_{k=0}^n (1-\frac{k}{n+1})^{1/2} z^k.$ Then for any $m\in\mathbb N,$ the family $\{A_{g_n}(m): n\in\mathbb N\}$ is not uniformly bounded.
\end{prop}
\begin{proof}
Fix $n\in\mathbb N.$ Let $c_k=(1-\frac{k}{n+1})^{1/2}, \ k=1,\ldots, n.$
Consider the sub-matrix $P A Q$ defined below of $D T_{g_n}(m)D^{-1},$ where
\[
A :=
\begin{pmatrix}
c_{s+1} - c_s & \ldots & c_{n+1} - c_n \\
\vdots & \ddots & \vdots \\
c_{s+1} - c_s & \ldots & c_{n+1} - c_n
\end{pmatrix}_{(s+1)\times (n-s+1)},
\]
 $$P=\mbox{diag}\bigg({{s\choose m-1}^\frac{1}{2}}, {{s+1\choose m-1}^\frac{1}{2}},\ldots, {{2s\choose m-1}^\frac{1}{2}}\bigg),$$ 
 $$Q=\mbox{diag}\bigg({s\choose m-1}^{-\frac{1}{2}}, {s+1\choose m-1}^{-\frac{1}{2}},\ldots, {n\choose m-1}^{-\frac{1}{2}}\bigg),$$ and 
 $$D=\mbox{diag}\bigg({{m-1\choose m-1}^\frac{1}{2}}, {{m\choose m-1}^\frac{1}{2}},\ldots\bigg).$$ 
 Note that 
\[\|PAQ\|\geqslant \alpha \beta \|A\|,\]
where $\alpha$ and $\beta$ are minimum of the eigenvalues of $P$ and $Q$ respectively, and $\|\cdot\|$ denotes the standard operator norm.
Using the argument as presented in \cite[Theorem 2.2]{CesaroM}, this in turn implies that
\begin{eqnarray}\label{lower bound on operator norm}
\|PAQ\|\geqslant \sqrt{\frac{{s\choose m-1}}{{n\choose m-1}}} \sqrt{s+1}\big(\sum_{k=s}^n |c_{k+1} - c_k|^2)^{1/2}
\end{eqnarray}
Since ${s\choose m-1}\sim (s+1)^{m-1}$ and ${n\choose m-1}\sim (n+1)^{m-1}$, it follows from \eqref{lower bound on operator norm} that
\begin{eqnarray*}
\|PAQ\|\geqslant \sqrt{\frac{(s+1)^m}{(n+1)^{m-1}}}\big(\sum_{k=s}^n |c_{k+1} - c_k|^2)^{1/2}.
\end{eqnarray*}
 It follows from \cite[page 7]{CesaroM} that
\[\|PAQ\|\geqslant \sqrt{\frac{(s+1)^m}{(n+1)^{m-1}}}\frac{1}{2\sqrt{n+1}}\sqrt{log(n+2-s)}=\frac{1}{2}\sqrt{\frac{(s+1)^m}{(n+1)^{m}}}\sqrt{log(n+2-s)}.\]
Choosing $s=\lfloor{\frac{n}{2}}\rfloor$, we get that 
\begin{eqnarray}\label{log-lower-bound}
\|PAQ\|\geqslant \frac{1}{2^{m+1}}\sqrt{log(2+n/2)}.
\end{eqnarray}
Now as $n$ approaches to $\infty$, the right hand side of \eqref{log-lower-bound} tends to $\infty$ as well.
This proves that the family $\{T_{g_n}(m):n\in\mathbb N\}$ is not uniformly bounded.
\end{proof}

\textbf{Proof of Theorem \ref{mainthm2}:} 
Let $\alpha=\frac{1}{2}$. Note that $\sigma_n^{\alpha}[f](z)=(h_n*f)(z),$ where $h_n(z)=\binom{n+\alpha}{\alpha}^{-1} \sum_{k=0}^n \binom{n-k+\alpha}{\alpha} z^k$,  $n\in \mathbb N$. By \cite[Theorem 3.2]{CesaroM}, for any $f\in\mathcal H_{\lambda, m}$, the sequence $\{h_n* f\}$ does not converge to $f$ in $\mathcal H_{\lambda, m}$ if and only if the sequence $\{{\phi_n* f}\}$ does not converge to $f$ in $\mathcal H_{\lambda, m}$, where for each $n\in\mathbb N,$ $\phi_n(z)=\sum_{k=1}^{n}c_kz^k$ with $c_k=(1-\frac{k}{n+1})^{1/2}, \ k=1,\ldots, n.$  By Proposition \ref{Propo3.4}, the family $\{A_{\phi_n}(m)\}_{n\in\mathbb N}$ is not uniformly bounded. Thus, by Remark \ref{rem norm}, it follows that
the family $\{H_{\phi_n}\}_{n\in\mathbb N}$, where  $H_{\phi_n}:\hat{\mathcal H}_{\lambda,m}\to \hat{\mathcal H}_{\lambda,m}$ is given by $H_{\phi_n}(f)=\phi_n*f$, is not uniformly bounded. An application of the uniform boundedness principle now completes the proof.

\section{Convergence of Generalized Ces{\`a}ro sum}
In this section, we provide an alternative proof of Theorem \ref{mainthm1} using the method of induction. For any $f\in\mathcal O(\mathbb D),$ let $Lf$ be the function in $\mathcal O(\mathbb D)$ defined by
\begin{align*}
Lf(z)= \frac{f(z)-f(0)}{z},\,\,\,z\in\mathbb D.
\end{align*}
It is known that for any $\mu\in \mathcal M_+(\mathbb T)$ and $m\in\mathbb N,$ the inequality $D_{\mu,m}(Lf) \leqslant D_{\mu,m}(f)$ holds for every $f\in\mathcal H_{\mu,m},$ see \cite[Lemma 2.9]{GGR}. Moreover, from \cite[Lemma 2.10]{GGR} and \cite[Corollary 3.3]{GGR}, we have the following result which will be crucial for the proof of Theorem \ref{mainthm1} presented in this section. 
\begin{lemma}\label{backward shift identity}
Let $m\geqslant 1$ and $\mu\in \mathcal M_+(\mathbb T).$ Then for any function $f$ in $\mathcal H_{\mu,m}$, we have
\begin{eqnarray*}
\sum\limits_{k=1}^{\infty}D_{\mu,j}(L^kf)=D_{\mu,j+1}(f), \quad 0 \leqslant j\leqslant m-1.
\end{eqnarray*}
\end{lemma}
Now we start with the following proposition which describes a relationship between $L^j(\sigma^{\alpha}_n[f])$ and $\sigma^{\alpha}_{n-j}[L^jf]$ for any function $f\in\mathcal O(\mathbb D).$
\begin{prop}\label{L and Gsigma relation}
For every  $f\in\mathcal O(\mathbb D),$  $n\in \mathbb Z_{\geqslant 0},$ $\alpha \geqslant 0$ and $j\in\mathbb N$ we have 
\begin{align*}
L^j(\sigma^{\alpha}_n[f]) &=
\begin{cases}
\frac{\Gamma(n+1)\Gamma(n+\alpha-j+1)}{\Gamma(n-j+1)\Gamma(n+\alpha+1)}\,\sigma^{\alpha}_{n-j}[L^jf], & j\leqslant n,\\
0, &  j > n.
\end{cases}
\end{align*}
\end{prop}
\begin{proof}
Since for any $f\in \mathcal O(\mathbb D)$ and $n\in\mathbb Z_{\geqslant 0},$ $\sigma^{\alpha}_n[f]$ is a polynomial of degree at most $n$, it follows that $L^j(\sigma^{\alpha}_n[f])=0$ whenever $j>n.$ 
For the remaining case, let $f\in \mathcal O(\mathbb D)$. Note that
\begin{align*}
\sigma_n^{\alpha}[f](z)= \frac{\Gamma(n+1)}{\Gamma(n+\alpha+1)} \sum\limits_{k=0}^n \frac{\Gamma(n-k+\alpha+1)}{\Gamma(n-k+1)} a_k z^k,\;\;\;\;\;\;~~~n\in \mathbb Z_{\geqslant 0},\;\alpha \geqslant 0.
\end{align*}
Operating $L$ on both sides, we get
\begin{align*}
 L(\sigma_n^{\alpha}[f](z))= \frac{\Gamma(n+1)}{\Gamma(n+\alpha+1)}\sum\limits_{k=0}^{n-1} \frac{\Gamma(n-k+\alpha)}{\Gamma(n-k)} a_{k+1} z^k,\;\;\;\;\;\;~~~n\in \mathbb N,\;\alpha \geqslant 0.
\end{align*}
Hence it follows that
\begin{align*}
L(\sigma^{\alpha}_n[f]) &=
\begin{cases}
\frac{\Gamma(n+1)\Gamma(n+\alpha)}{\Gamma(n+\alpha+1)\Gamma(n)}\sigma^{\alpha}_{n-1}[Lf], & n\geqslant 1,\\
0, &   n=0.
\end{cases}
\end{align*}
This completes the proof of the proposition for the case $j=1$. 
Fix $k\in\mathbb N$ with $k<n.$ Assume that the statement of the proposition holds for $j= k.$ 
Then applying the induction hypothesis for the function $Lf$ we have that for any $m\in \mathbb Z_{\geqslant 0}$ satisfying $m\geqslant k,$
\begin{align*}
L^k (\sigma^{\alpha}_m[Lf])= \frac{\Gamma(m+1)\Gamma(m+\alpha-k+1)}{\Gamma(m-k+1)\Gamma(m+\alpha+1)}\,\sigma^{\alpha}_{m-k}[L^{k+1}f].
\end{align*}
Since $L(\sigma^{\alpha}_{m+1}[f])=\frac{\Gamma(m+2)\Gamma(m+\alpha+1)}{\Gamma(m+\alpha+2)\Gamma(m+1)}\, \sigma^{\alpha}_m[Lf]$ for every $m\in \mathbb Z_{\geqslant 0},$ we obtain that 
\begin{align*}
L^{k+1}(\sigma^{\alpha}_{m+1}[f])=  \frac{\Gamma(m+2)\Gamma(m+\alpha-k+1)}{\Gamma(m+\alpha+2)\Gamma(m-k+1)}\,\sigma^{\alpha}_{m-k}[L^{k+1}f],\,\,m\geqslant k. 
\end{align*}
This gives us that for any $n\in \mathbb Z_{\geqslant 0},$
\begin{align*}
L^{k+1}(\sigma^{\alpha}_n[f]) &=
\begin{cases}
\frac{\Gamma(n+1)\Gamma(n+\alpha-k)}{\Gamma(n+\alpha+1)\Gamma(n-k)}\,\sigma^{\alpha}_{n-k-1}[L^{k+1}f], & k+1\leqslant n,\\
0, &  k+1 > n.
\end{cases}
\end{align*}
This finishes the induction step for $j=k+1$ and completes the proof of the proposition.
\end{proof}

\textbf{Alternative proof of Theorem \ref{mainthm1}:} Let $\mu \in \mathcal M_{+}(\mathbb T),$ $n\in\mathbb N,$ and $\alpha > \frac{1}{2}.$
We will prove this theorem by induction on $m.$ From \cite[Theorem 1.1]{CesaroM}, it follows that there exists a constant $M_{\alpha},$ independent of $n,$ such that $D_{\mu,1} (\sigma^{\alpha}_n[f]) \leqslant M_{\alpha} D_{\mu,1} (f)$ for every $f\in\mathcal H_{\mu,1}.$ Let us assume that $D_{\mu,m} (\sigma^{\alpha}_n[f]) \leqslant M_\alpha D_{\mu,m} (f)$ for every $f\in \mathcal H_{\mu,m},$ and for $m=1,\ldots,k.$ Now take $f\in \mathcal H_{\mu,k+1}.$ Since $\sigma^{\alpha}_n[f]$ is a polynomial of degree at most $n$, from Lemma \ref{backward shift identity}, it follows that
\begin{align*}
D_{\mu,k+1} (\sigma^{\alpha}_n[f]) = \sum\limits_{j=1}^n D_{\mu,k} (L^j (\sigma^{\alpha}_n[f]).
\end{align*}
Now applying the Proposition \ref{L and Gsigma relation} and the induction hypothesis we obtain that 
\begin{align*}
D_{\mu,k+1} (\sigma^{\alpha}_n[f])  &= \sum\limits_{j=1}^n \left( \frac{\Gamma(n+1)\Gamma(n+\alpha-j+1)}{\Gamma(n-j+1)\Gamma(n+\alpha+1)}\right)^2 D_{\mu,k} (\sigma^{\alpha}_{n-j}[L^jf])\\
&\leqslant \sum\limits_{j=1}^n D_{\mu,k} (\sigma^{\alpha}_{n-j}[L^jf])\\
& \leqslant M_{\alpha}\sum\limits_{j=1}^n D_{\mu,k} (L^jf)\\
& \leqslant M_{\alpha} \sum\limits_{j=1}^{\infty} D_{\mu,k} (L^jf)\\
&= M_{\alpha}D_{\mu,k+1} (f).
\end{align*} 
This completes the induction step for $m=k+1$ and as well as the proof of the theorem.


\part*{\textbf{Declarations}}
\noindent Conflict of interest : The authors declare that there are no conflict of interest regarding the publication of this paper.


\end{document}